\documentclass[12pt]{amsart}
\usepackage{amsmath,amssymb,txfonts}
\usepackage{amssymb}
\usepackage{amsxtra}
\usepackage{amsmath}
\usepackage{txfonts}
\newtheorem{lemma}{Lemma}
\newtheorem{question}{Question}
\newtheorem{theorem}{Theorem}
\newtheorem{corollary}{Corollary}
\newtheorem{conjecture}{Conjecture}
\newtheorem{remark}{Remark}

\def\bl{\begin{lemma}}
\def\bt{\begin{theorem}}
\def\el{\end{lemma}}
\def\et{\end{theorem}}
\def\bp{\begin{proof}}
\def\ep{\end{proof}}
\def\bc{\begin{corollary}}
\def\ec{\end{corollary}}
\def\bc{\begin{remark}}
\def\ec{\end{remark}}
\subjclass[2000]{30H10, 30H25, 30H35, 46E15, 47B33} 

\begin{document}

\title[Composition Operators between Analytic Campanato Spaces]{Composition Operators between Analytic Campanato Spaces}
\author{Jie Xiao and Wen Xu}
\address{Department of Mathematics \& Statistics,
         Memorial University,
         NL A1C 5S7, Canada}
         \email{jxiao@mun.ca; wenxupine@gmail.com}
\thanks{JX and WX were in part supported by NSERC of Canada and the Finnish Cultural Foundation, respectively.}

\begin{abstract}
This note characterizes both boundedness and compactness of a composition operator between any two analytic Campanato spaces on the unit complex disk.
\end{abstract}
 \maketitle
\section{Introduction}\label{s1}
\setcounter{equation}{0}

  On the basis of the works: \cite{Smith}, \cite{Jar}, \cite{Lait1, Lait2}, \cite{LaitNST}, \cite{Wulan}, \cite{WulanZZ}, \cite{Tjani}, \cite{BourdonCM}, \cite{Madigan, MadiganM}, \cite{Roan}, \cite{GauthierX}, \cite{Xiao, Xiao1, Xiao2}, \cite{Shapiro1, Shapiro} and \cite{CaFO}, we consider an unsolved fundamental problem in the function-theoretic operator theory, i.e., the so-called composition operator question for the analytic Campanato spaces:

 \begin{question}\label{q} Let $\phi$ be an analytic self-map of $\mathbb D$ and $-\infty<p,q<\infty$. What finite (resp. vanishing) property must $\phi$ have in order that $C_\phi$ is bounded (resp. compact) between $\mathcal{CA}_p$ and $\mathcal{CA}_q$ ?
 \end{question}

In the above and below, $\mathbb D$ and $\mathbb T$ respectively represent the unit disk and the unit circle in the finite complex plane $\mathbb C$, $C_\phi f=f\circ\phi$ is the composition of an analytic function $f$ on $\mathbb D$ with $\phi$, and for $p\in (-\infty,\infty)$, and $\mathcal{CA}_p$ denotes the so-called Campanato space of all analytic functions $f: \mathbb D\to\mathbb C$ with radial boundary values $f$ on $\mathbb T$ satisfying
$$
\|f\|_{\mathcal{CA}_p}=\sup_{I\subseteq\mathbb T}\sqrt{|I|^{-p}\int_I|f(\xi)-f_I|^2|d\xi|}<\infty
$$
where the supremum is taken over all sub-arcs $I\subseteq\mathbb T$ with $|I|$ being their arc-lengths, and
$$
|d\zeta|=|de^{i\theta}|=d\theta;\quad f_I=|I|^{-1}\int_I f(\xi)\,|d\xi|.
$$
Neededless to say, $\|\cdot\|_{\mathcal{CA}_p}$ cannot distinguish between any two $\mathcal{CA}_p$ functions differing by a constant, but $|f(0)|+\|\cdot\|_{\mathcal{CA}_p}$ defines a norm so that $\mathcal{CA}_p$ is a Banach space. Here, it is perhaps appropriate to mention the following table which helps us get a better understanding of the structure of $\mathcal{CA}_p$ (see, e.g. \cite[pp. 67-75]{Giaq} and \cite[p. 52]{Xiao2}):

\begin{center}
    \begin{tabular}{ | l | p{7cm} |}
    \hline
    Index $p$ & Analytic Campanato Space $\mathcal{CA}_p$\\ \hline
    $p\in (-\infty,0]$ & Analytic Hardy space $\mathcal H^2$ \\ \hline
    $p\in (0,1)$ & Holomorphic Morrey space $\mathcal H^{2,p}$\\ \hline
    $p=1$ & Analytic John-Nirenberg space $\mathcal{BMOA}$ \\ \hline
    $p\in (1,3]$ & Analytic Lipschitz space $\mathcal A_{\frac{p-1}{2}}$\\ \hline
    $p\in (3,\infty)$ & Complex constant space $\mathbb C$\\ \hline
\end{tabular}
\end{center}

An answer to the boundedness part of Question \ref{q} is the following result.

\begin{theorem}\label{tA} Let $\phi$ be an analytic self-map of $\mathbb D$ and $(p,q)\in [0,2)\times [0,2)$. Then $C_\phi: \mathcal{CA}_p\mapsto\mathcal{CA}_q$ is bounded if and only if

\begin{equation}\label{e1}
\sup_{a\in\mathbb D}\frac{(1-|a|^2)^{1-q}}{(1-|\phi(a)|^2)^{1-p}}\|\sigma_{\phi(a)}\circ\phi\circ\sigma_a\|_2^2<\infty,
\end{equation}
where
$$
\sigma_b(z)=\frac{b-z}{1-\bar{b}z}\quad\&\quad \|f\|_2=\sqrt{\int_{\mathbb T}|f(\xi)|^2|d\xi|}.
$$
\end{theorem}

It should be pointed out that (\ref{e1}) is not always true - in fact, we have the following consequence whose (i) with $p=q\in\{0,1\}$ and (ii) are well-known; see e.g. \cite{Madigan, Roan, Shapiro1, Shapiro, Smith, Xiao}.

\begin{corollary}\label{cor1} Let $\phi$ be an analytic self-map of $\mathbb D$. For $f\in\mathcal{H}^2$ and $p\in [0,2)$ set
$$
\|f\|_{\mathcal{CA}_p,\ast}=\sup_{a\in\mathbb D}(1-|a|^2)^{\frac{1-p}{2}}\|f\circ\sigma_a-f(a)\|_2.
$$
\item{\rm(i)} If $p\in [0,1]$ then $C_\phi: \mathcal{CA}_p\mapsto\mathcal{CA}_p$ is always bounded with
\begin{equation}\label{e1a}
\|C_\phi f\|_{\mathcal{CA}_p,\ast}\le \left(\frac{1+|\phi(0)|}{1-|\phi(0)|}\right)^{\frac{1-p}{2}}\|f\|_{\mathcal{CA}_p,\ast}.
\end{equation}

\item{\rm(ii)} If $p\in (1,2)$ then $C_\phi: \mathcal{CA}_p\mapsto\mathcal{CA}_p$ is bounded when and only when

\begin{equation}\label{e1b}
\sup_{a\in\mathbb D}\left(\frac{1-|a|^2}{1-|\phi(a)|^2}\right)^{\frac{3-p}{2}}|\phi'(a)|<\infty.
\end{equation}

\end{corollary}

Below is a partial answer to the compactness part of Question \ref{q}.

\begin{theorem}\label{tB} Let $\phi$ be an analytic self-map of $\mathbb D$ and $(p,q)\in [0,2)\times[0,2)$. If $C_\phi: \mathcal{CA}_p\mapsto\mathcal{CA}_q$ is compact then (\ref{e1}) holds and
\begin{equation}\label{e2}
\lim_{|\phi(a)|\to 1}\frac{(1-|a|^2)^{1-q}}{(1-|\phi(a)|^2)^{1-p}}\|\sigma_{\phi(a)}\circ\phi\circ\sigma_a\|_2^2=0.
\end{equation}
Conversely, if (\ref{e1}) holds and (\ref{e2}) is valid for $(p,q)\in [0,2)\times[1,1]\cup (1,2)\times [0,2)$ then $C_\phi: \mathcal{CA}_p\mapsto\mathcal{CA}_q$ is compact.
\end{theorem}

Theorem \ref{tB} covers the corresponding $\mathcal{BMOA}$-results in \cite{Smith, WulanZZ, LaitNST}, but also it derives the following assertion extending the known one in \cite{Roan, MadiganM, Xiao}.

\begin{corollary}\label{cor2}
Let $\phi$ be an analytic self-map of $\mathbb D$ and $p\in [0,2)$. If $C_\phi: \mathcal{CA}_p\mapsto\mathcal{CA}_p$ is compact then (\ref{e1b}) holds and
\begin{equation}\label{e1aa}
\lim_{|\phi(a)|\to 1}\left(\frac{1-|a|^2}{1-|\phi(a)|^2}\right)^{\frac{3-p}{2}}|\phi'(a)|=0.
\end{equation}
Conversely, if (\ref{e1b}) holds and (\ref{e1aa}) is valid for $p\in (1,2)$ then $C_\phi: \mathcal{CA}_p\mapsto\mathcal{CA}_p$ is compact.
\end{corollary}

\begin{conjecture}\label{c} The converse part of Theorem \ref{tB} still holds for $(p,q)\in [0,2)\times[0,2)\setminus \big([0,2)\times[1,1]\cup (1,2)\times[0,2)\big)$.
\end{conjecture}

{\it Notation}: From now on, ${\mathsf X}\lesssim{\mathsf Y}$,
${\mathsf X}\gtrsim{\mathsf Y}$, and ${\mathsf X}\approx{\mathsf Y}$
represent that there exists a constant $\kappa>0$ such that ${\mathsf
X}\le \kappa{\mathsf Y}$, ${\mathsf X}\ge \kappa{\mathsf Y}$, and
$\kappa^{-1}{\mathsf Y}\le{\mathsf X}\le \kappa{\mathsf Y}$, respectively. In addition, $dm$ stands for two dimensional Lebesgue measure.

\section{Boundedness}\label{s2}

In order to prove Theorem \ref{tA} and Corollary \ref{cor1}, we need two lemmas.

\begin{lemma}\label{l21} Let $p\in [0,2)$ and $f\in\mathcal{H}^2$. Then $f\in\mathcal{CA}_p$ if and only if $\|f\|_{\mathcal{CA}_p,\ast}<\infty$.
\end{lemma}
\begin{proof} {\it Case 1}: $p=0$. This is trivial.

{\it Case 2}: $p\in (0,1]$. This situation can be verified by \cite[Theorem 3.2.1]{Xiao2} and the well-known Hardy-Littlewood identity for $f\in \mathcal{H}^2$:
\begin{equation}\label{e4}
\pi^{-1}\int_{\mathbb D}|f'(z)|^2(-\ln|z|^2)\,dm(z)=(2\pi)^{-1}\int_{\mathbb T}|f(\xi)-f(0)|^2\,|d\xi|.
\end{equation}

{\it Case 3}: $p\in (1,2)$. Let $g=f\circ\sigma_a-f(a)$. Then
\begin{equation}\label{e5}
(1-|a|^2)|f'(a)|=|g'(0)|\le (2\pi)^{-1/2}\|g\|_2=(2\pi)^{-1/2}\|f\circ\sigma_a-f(a)\|_2.
\end{equation}
If $f\in\mathcal{CA}_p$ then an application of (\ref{e5}) yields
\begin{equation}\label{e5a}
\sup_{a\in\mathbb D}(1-|a|^2)^{\frac{3-p}{2}}|f'(a)|<\infty
\end{equation}
and consequently, $f\in\mathcal{A}_{\frac{p-1}{2}}$, as desired. Conversely, if $f\in\mathcal{A}_{\frac{p-1}{2}}$ then
$$
\mathsf{A}=\sup_{\xi_1\not=\xi_2\ in\ {\mathbb D}\cup\mathbb T}\frac{|f(\xi_1)-f(\xi_2)|}{|\xi_1-\xi_2|^{\frac{p-1}{2}}}<\infty.
$$
This, along with $p\in (1,2)$ and \cite[p. 63, Ex. 8]{Zhu}, gives
\begin{eqnarray*}
&&\|f\|^2_{\mathcal{CA}_p,\ast}\\
&&=\sup_{a\in\mathbb D}(1-|a|^2)^{1-p}\int_{\mathbb T}|f\circ\sigma_a(\xi)-f(a)|^2\,|d\xi|\\
&&\lesssim\mathsf{A}^2\sup_{a\in\mathbb D}\int_{\mathbb T}\left(\frac{1-|a|^2}{|\sigma_a(\xi)-a|}\right)^{1-p}\,|d\xi|\\
&&\approx\mathsf{A}^2\sup_{a\in\mathbb D}\int_{\mathbb T}\frac{(1-|a|^2)^{2-p}}{|1-\bar{a}\eta|^{3-p}}\,|d\eta|\\
&&\lesssim \mathsf{A}^2.
\end{eqnarray*}
\end{proof}

\begin{lemma}\label{l22} For $p\in [0,2)$ let $f_b(z)=(1-|b|^2)^{\frac{1+p}{2}}/(1-\bar{b}z)$. Then $f_b$ is uniformly bounded in $\mathcal{CA}_p$, i.e., $\sup_{b\in\mathbb D}\|f_b\|_{\mathcal{CA}_p,\ast}<\infty.$
\end{lemma}
\begin{proof} Using \cite[Lemma 2.5]{OrtF}, we get the following estimate:
\begin{eqnarray*}
&&\mathsf{B}\\
&&=\int_{\mathbb D}|f'_b(z)|^2(1-|\sigma_a(z)|^2)\,dm(z)\\
&&=\big(|b|(1-|b|^2)^\frac{1+p}{2}\big)^2(1-|a|^2)\int_{\mathbb D}\frac{1-|z|^2}{|1-\bar{a}z|^2|1-\bar{b}z|^4}\,dm(z)\\
&&\lesssim\frac{\big(|b|(1-|b|^2)^\frac{1+p}{2}\big)^2(1-|a|^2)}{(1-|b|^2)|1-\bar{a}b|^2}.
\end{eqnarray*}
Choosing $a=\sigma_b(c)$, we utilize $1-|z|\lesssim -\ln |z|$ to obtain that if $p\in [0,2)$ then
\begin{eqnarray*}
&&\|f_b\|_{\mathcal{CA}_p,\ast}^2\\
&&\lesssim\sup_{a\in\mathbb D}(1-|a|^2)^{1-p}\mathsf{B}\\
&&\lesssim\sup_{a\in\mathbb D}\frac{(1-|a|^2)^{2-p}(1-|b|^2)^{p}}{|1-\bar{a}b|^2}\\
&&=\sup_{c\in\mathbb D}\left(\frac{1-|c|^2}{|1-\bar{b}c|}\right)^{2-p}|1-\bar{b}c|^{p}\\
&&\le 2^2,
\end{eqnarray*}
as desired.
\end{proof}

\begin{proof}[Proof of Theorem \ref{tA}] Using \cite[Proposition 2.3]{Lait1}, we have that if $g(0)=0=\psi(0)$, $g\in\mathcal{H}^2$, and $\psi$ is an analytic self-map of $\mathbb D$, then
\begin{equation}\label{e8}
\|g\circ\psi\|_2\lesssim\|g\|_2\|\psi\|_2.
\end{equation}
Setting
$$
g_a=f\circ\sigma_{\phi(a)}-f\circ\phi(a)\quad\&\quad \psi_a=\sigma_{\phi(a)}\circ\phi\circ\sigma_a,
$$
we get
$$
g_a\circ\psi_a=f\circ\phi\circ\sigma_a-f\circ\phi(a).
$$
As a consequence of Lemma \ref{l21} and (\ref{e8}), we find that if (\ref{e1}) is valid then
\begin{eqnarray*}
&&\|C_\phi f\|_{\mathcal{CA}_q,\ast}^2\\
&&=\sup_{a\in\mathbb D}(1-|a|^2)^{1-q}\|f\circ\phi\circ\sigma_a-f\circ\phi(a)\|_2^2\\
&&\lesssim\sup_{a\in\mathbb D}\frac{(1-|a|^2)^{1-q}}{(1-|\phi(a)|^2)^{1-p}}\big((1-|\phi(a)|^2)^{1-p}\|g_a\|_2^2\big)\|\psi_a\|_2^2\\
&&\lesssim\|f\|_{\mathcal{CA}_p,\ast}^2\sup_{a\in\mathbb D}\frac{(1-|a|^2)^{1-q}}{(1-|\phi(a)|^2)^{1-p}}\|\psi_a\|_2^2,
\end{eqnarray*}
and consequently, $C_\phi$ exists as a bounded operator from $\mathcal{CA}_p$ into $\mathcal{CA}_q$.

For the ``only-if" part, recall the so-called Navanlinna counting function of $\phi$:
$$
N(\phi,w)=\sum_{z:\ \phi(z)=w}\ln |z|^{-1}\quad\forall\quad w\in\mathbb D\setminus\{\phi(0)\}
$$
and the associated change of variable formula:
\begin{equation}\label{e9}
\int_{\mathbb D}|(C_\phi f)'(z)|^2\ln |z|^{-1}\,dm(z)=\int_{\mathbb D}|f'(w)|^2 N(\phi,w)\,dm(w)\ \ \forall\ f\in\mathcal{H}^2.
\end{equation}
A combination of (\ref{e9}) and (\ref{e4}) gives that if $b=\phi(a)$ then
\begin{equation}\label{e10}
\|\sigma_b\circ\phi\circ\sigma_a\|_2^2=4\int_{\mathbb D}N(\sigma_b\circ\phi\circ\sigma_a,z)\,dm(z).
\end{equation}
Now, if $C_\phi: \mathcal{CA}_p\mapsto\mathcal{CA}_q$ is bounded, then the test function $f_b$ in Lemma \ref{l22} is used to imply
$$
\mathsf{C}=\sup_{a,b\in\mathbb D}(1-|a|^2)^{1-q}\int_{\mathbb D}|(f_b\circ\phi\circ\sigma_a)'(z)|^2\ln|z|^{-1}\,dm(z)<\infty.
$$
and consequently,
\begin{eqnarray*}
&&\mathsf{C}\\
&&\gtrsim\sup_{a,b\in\mathbb D}(1-|a|^2)^{1-q}|b|^2(1-|b|^2)^{p-1}\int_{\mathbb D}|\sigma_b'(z)|^2N(\phi\circ\sigma_a,z)\,dm(z)\\
&&\gtrsim\sup_{a\in\mathbb D}\frac{(1-|a|^2)^{1-q}|\phi(a)|^2}{(1-|\phi(a)|^2)^{1-p}}\int_{\mathbb D}N(\sigma_{\phi(a)}\circ\phi\circ\sigma_a,z)\,dm(z)\\
&&\gtrsim\sup_{a\in\mathbb D}\frac{(1-|a|^2)^{1-q}|\phi(a)|^2}{(1-|\phi(a)|^2)^{1-p}}\|\sigma_{\phi(a)}\circ\phi\circ\sigma_a\|_2^2\\
&&\gtrsim s^2\sup_{a\in\mathbb D}\frac{(1-|a|^2)^{1-q}}{(1-|\phi(a)|^2)^{1-p}}\|\sigma_{\phi(a)}\circ\phi\circ\sigma_a\|_2^2
\end{eqnarray*}
as $|\phi(a)|> s\in (0,1)$. Note also that the identity map $f(z)=z$ is an element of $\mathcal{CA}_p$. Thus, boundedness of $C_\phi: \mathcal{CA}_p\mapsto\mathcal{CA}_q$ ensures $\|\phi\|_{\mathcal{CA}_q,\ast}<\infty$, and consequently, if $|\phi(a)|\le s<1$ then
\begin{eqnarray*}
&&\sup_{a\in\mathbb D}\frac{(1-|a|^2)^{1-q}}{(1-|\phi(a)|^2)^{1-p}}\|\sigma_{\phi(a)}\circ\phi\circ\sigma_a\|_2^2\\
&&\lesssim\big(1+(1-s)^{p-1}\big)\sup_{a\in\mathbb D}{(1-|a|^2)^{1-q}}\int_{\mathbb T}\left|\frac{{\phi(a)}-\phi\circ\sigma_a(\xi)}{1-\overline{\phi(a)}\phi\circ\sigma_a(\xi)}\right|^2\,|d\xi|\\
&&\lesssim\Big(\frac{1+(1-s)^{p-1}}{(1-s)^2}\Big)\sup_{a\in\mathbb D}{(1-|a|^2)^{1-q}}\|{\phi(a)}-\phi\circ\sigma_a\|_2^2\\
&&\approx\Big(\frac{1+(1-s)^{p-1}}{(1-s)^2}\Big)\|\phi\|_{\mathcal{CA}_q,\ast}^2.
\end{eqnarray*}
The above estimates imply (\ref{e1}).

\end{proof}

\begin{proof}[Proof of Corollary \ref{cor1}] (i) Under $p\in [0,1]$, we use the Schwarz lemma for $\sigma_{\phi(0)}\circ\phi$ to deduce that (1) holds for $p=q\in [0,1]$, and so that $C_\phi$ is bounded on $\mathcal{CA}_p$ due to Theorem \ref{tA}. To reach (\ref{e1a}), let us begin with the case $\phi(0)=0$. According to the setting in the argument for Theorem \ref{tA}, the well-known Littlewood subordination principle and Schwarz's lemma for $\phi$, we have
\begin{eqnarray*}
&&\|C_\phi f\|^2_{\mathcal{CA}_p,\ast}\\
&&=\sup_{a\in\mathbb D}(1-|a|^2)^{1-p}\|g_a\circ\psi_a\|_2^2\\
&&\le\sup_{a\in\mathbb D}(1-|a|^2)^{1-p}\|g_a\|_2^2\\
&&\le\|f\|^2_{\mathcal{CA}_p,\ast}\sup_{a\in\mathbb D}\left(\frac{1-|a|^2}{1-|\phi(a)|^2}\right)^{1-p}\\
&&\le\|f\|^2_{\mathcal{CA}_p,\ast}.
\end{eqnarray*}

Next, for the general case let
$$
\begin{cases}
\psi=\sigma_{\phi(0)}\circ\phi;\\
\lambda=\frac{\bar{a}b-1}{1-a\bar{b}};\\
b=\phi(0);\\
c=\sigma_a(b).
\end{cases}
$$
Then $\psi(0)=0$ and thus
\begin{eqnarray*}
&&\|C_{\sigma_b}f\|^2_{\mathcal{CA}_p,\ast}\\
&&=\sup_{a\in\mathbb D}(1-|a|^2)^{1-p}\|f(\lambda\sigma_c)-f(\lambda c)\|_2^2\\
&&\le\|f\|_{\mathcal{CA}_p,\ast}^2\sup_{a\in\mathbb D}\left(\frac{1-|a|^2}{1-|c|^2}\right)^{1-p}\\
&&\le\|f\|_{\mathcal{CA}_p,\ast}^2\left(\frac{1+|b|}{1-|b|}\right)^{1-p}.
\end{eqnarray*}
Using the previous estimates, we get
$$
\|C_\phi f\|^2_{\mathcal{CA}_p,\ast}=\|f\circ\sigma_b\circ \psi\|^2_{\mathcal{CA}_p,\ast}\le\|f\circ\sigma_b\|_{\mathcal{CA}_p,\ast}^2
\le\|f\|_{\mathcal{CA}_p,\ast}^2\left(\frac{1+|b|}{1-|b|}\right)^{1-p},
$$
whence reaching (\ref{e1a}).

(ii) Suppose $p\in (1,2)$. Then (\ref{e5}) yields
$$
\sup_{a\in\mathbb D}\left(\frac{1-|a|^2}{1-|\phi(a)|^2}\right)^{3-p}|\phi'(a)|^2\le\sup_{a\in\mathbb D}\left(\frac{1-|a|^2}{1-|\phi(a)|^2}\right)^{1-p}\|\sigma_{\phi(a)}\circ\phi\circ\sigma_a\|_2^2,
$$
so, if $C_\phi$ is bounded on $\mathcal{CA}_p$ then (\ref{e1}) holds with $p=q$ due to Theorem \ref{tA}, and hence (\ref{e1b}) holds. Conversely, if (\ref{e1b}) is true, then $C_\phi$ is bounded on $\mathcal{A}_{\frac{p-1}{2}}$ (cf. \cite[Theorem A]{Madigan}) and hence bounded on $\mathcal{CA}_p$.
\end{proof}

\section{Compactness}\label{s3}

The arguments for Theorem \ref{tB} and Corollary \ref{cor2} depend on the two basic facts below.

\begin{lemma}\label{l32} Let $p\in (0,2)$ and $f\in\mathcal{CA}_p$ with $f(0)=0$. Then
$$
\int_{\mathbb T}|f(\xi)|^4\,|d\xi|\lesssim \begin{cases}\|f\|_{\mathcal{CA}_p,\ast}^2\|f\|_2^2\ \ \hbox{for}\ \ p\in [1,2);\\
\|f\|_{\mathcal{CA}_p,\ast}^2\int_0^\infty tH^p_\infty\big(\{\xi\in\mathbb T:\ |f(\xi)|>t\}\big)\,dt\ \ \hbox{for}\ \ p\in (0,1],
\end{cases}
$$
where $H^p_\infty(E)=\inf_{E\subseteq\cup_{I_j}}\sum_j|I_j|^p$ is $p$-dimensional Hausdorff capacity of $E\subseteq\mathbb T$ - the infimum is taken over all arc coverings $\cup_j I_j\supseteq E$.
\end{lemma}
\begin{proof} Let $d\mu=|f'(z)|^2(1-|z|^2)dm(z)$. From $f\in\mathcal{CA}_p$ it follows that $\mu$ is a $p$-Carleson measure on $\mathbb D$ - in other words -
$$
\|\mu\|_{\mathcal{CM}_p}=\sup_{I\subseteq\mathbb T}|I|^{-p}\mu\big(S(I)\big)\lesssim\|f\|_{\mathcal{CA}_p,\ast}^2,
$$
where $
S(I)=\{z=re^{i\theta}\in\mathbb D:\ 1-|I|/(2\pi)\le r<1\ \&\ |\theta-\theta_I|\le |I|/2\}$
is the Carleson box based on the arc $I\subseteq\mathbb T$ taking $\theta_I$ as its center. In fact, if $a=(1-|I|/(2\pi))e^{i(\theta_I+|I|/4)}$ then a simple computation, along with (\ref{e4}) and $-\ln |z|\approx 1-|z|^2$ as $|z|\ge 2^{-1}$ as well as Lemma \ref{l21}, gives
\begin{eqnarray*}
&&|I|^{-p}\mu\big(S(I)\big)\\
&&\lesssim (1-|a|^2)^{1-p}\int_{S(I)}|f'(z)|^2\big(1-|\sigma_a(z)|^2\big)\,dm(z)\\
&&\lesssim(1-|a|^2)^{1-p}\|f\circ\sigma_a-f(a)\|_2^2\\
&&\lesssim\|f\|^2_{\mathcal{CA}_p,\ast}.
\end{eqnarray*}
In particular, when $p\in (1,2)$, $\mu$ is also $1$-Carleson measure with
$\|\mu\|_{\mathcal{CM}_1}\lesssim\|\mu\|_{\mathcal{CM}_p}$. According to \cite[p. 79, Theorem 4.1.4]{Xiao2}, we have
$$
\int_{\mathbb D}|f|^2\,d\mu\lesssim \|\mu\|_{\mathcal{CM}_p}\begin{cases}
& \|f\|_2^2\ \ \hbox{for}\ p\in [1,2);\\
& \int_0^\infty tH^p_\infty\big(\{\xi\in\mathbb T:\ |f(\xi)|>t\}\big)\,dt\ \ \hbox{for}\ p\in (0,1].
\end{cases}
$$
This last estimate, along with the following Hardy-Stein identity based estimate  (cf. \cite[p. 36]{Xiao2})
\begin{eqnarray*}
&&\int_{\mathbb T}|f(\xi)|^4\,|d\xi|\\
&&\approx\int_{\mathbb D}|f(z)|^2|f'(z)|^2(\ln|z|^{-1})\,dm(z)\\
&&\lesssim\int_{\mathbb D}|f(z)|^2|f'(z)|^2(1-|z|^2)\,dm(z)\\
&&\approx\int_{\mathbb D}|f|^2\,d\mu,
\end{eqnarray*}
implies the desired estimate.
\end{proof}

\begin{lemma}\label{l31} Let $(p,q)\in [0,2)\times[1,1]$. If an analytic self-map $\phi$ of $\mathbb D$ satisfies (\ref{e2}), then one has
$$
\lim_{t\to 1}\sup_{|\phi(a)|\le s}\frac{(1-|a|^2)^{1-q}}{(1-|\phi(a)|^2)^{1-p}}|\{\xi\in\mathbb T:\ |\sigma_{\phi(a)}\circ\phi\circ\sigma_a(\xi)|>t\}|=0\ \forall\ s\in (0,1).
$$
\end{lemma}

\begin{proof} Note that
$$
|\phi\circ\sigma_a|\to 1\Longleftrightarrow|\sigma_{\phi(a)}\circ\phi\circ\sigma_a|\to 1\quad\hbox{under}\quad|\phi(a)|\le s.
$$
So, it suffices to show that (\ref{e2}) implies
\begin{equation}\label{eee}
\lim_{t\to 1}\sup_{|\phi(a)|\le s}\frac{(1-|a|^2)^{1-q}}{(1-|\phi(a)|^2)^{1-p}}|\{\xi\in\mathbb T:\ |\phi\circ\sigma_a(\xi)|>t\}|=0\ \forall\ s\in (0,1).
\end{equation}

Following \cite{LaitNST}, for $re^{i\theta}\in\mathbb D$ let
$$
J(re^{i\theta})=\{e^{it}:\ |t-\theta|\le\pi(1-r)\}.
$$
Clearly, $J(re^{i\theta})$ is the sub-arc of $\mathbb T$ centered at $e^{i\theta}$. Importantly, \cite[Lemma 3]{LaitNST} tells us that for any measurable set $E\subseteq\mathbb T$ with $1$-dimensional Lebesgue measure $|E|>0$ there exists a measurable set $F\subseteq E$ such that $|F|>0$ and
\begin{equation}\label{e4e}
\frac{|J(r\xi)\cap E|}{|J(r\xi)|}\ge{(2^{4}\pi)^{-1}|E|}\quad\forall\quad r\in [0,1)\ \&\ \xi\in F.
\end{equation}

Suppose now (\ref{e2}) is valid but (\ref{eee}) is not true. On the one hand, we have that for any $\epsilon>0$ there is an $s\in (0,1)$ such that
\begin{equation}\label{eYe}
\frac{\big(\frac{2\pi}{|J(a)|}\big)\int_{J(a)}\rho\big(\phi\circ\sigma_b(\xi),\phi\circ\sigma_b(a)\big)^2\,{|d\xi|}}{
{(1-|\phi\circ\sigma_b(a)|^2)^{1-p}}{(1-|\sigma_b(a)|^2)^{q-1}}}<\epsilon\ \ \forall\ \ |\phi\circ\sigma_b(a)|> s.
\end{equation}
Here we have used the pseudo-hyperbolic distance $\rho(z,w)=|\sigma_w(z)|$ between $z,w\in\mathbb D$ and the following basic estimate
$$
\begin{cases}
\|\sigma_{\phi(a)}\circ\phi\circ\sigma_a\|_2^2=\int_{\mathbb T}\rho\big(\phi(\xi),\phi(a)\big)^2\,P_a(\xi)|d\xi|;\\
P_a(\xi)=|\sigma_a'(\xi)|\ge 2^{-1}\pi |J(a)|^{-1}\ \forall\ \xi\in J(a).
\end{cases}
$$
On the other hand, we can select two constants $s_0\in (0,1)$ and $\epsilon_0>0$, points $b_j\in\mathbb D$, and numbers $t_j\in (0,1)$ with $\lim_{j\to\infty}t_j=1$ such that for any $j=1,2,...$ one has $|\phi(b_j)|\le s_0$ and
$$
E_j=\{\xi\in\mathbb T:\ \phi_j(\xi)=\phi\circ\sigma_{b_j}(\xi)\ \hbox{exists\ as\ radial\ limit\ and}\ |\phi_j(\xi)|>t_j\}
$$
obeys
\begin{equation}\label{e5e}
\left(\frac{(1-|b_j|^2)^{1-q}}{(1-|\phi(b_j)|^2)^{1-p}}\right)(2\pi)^{-1}|E_j|\ge\epsilon_0.
\end{equation}
This (\ref{e5e}), plus the above-stated lemma on (\ref{e4e}), ensures that one can choose sets $F_j\subseteq E_j$ such that $|F_j|>0$ and
\begin{eqnarray}\label{e6e}
&&\left(\frac{(1-|b_j|^2)^{1-q}}{(1-|\phi(b_j)|^2)^{1-p}}\right)\frac{|J(r\xi)\cap E_j|}{|J(r\xi)|}\nonumber\\
&&\ge\left(\frac{(1-|b_j|^2)^{1-q}}{(1-|\phi(b_j)|^2)^{1-p}}\right)\Big(\frac{|E_j|}{2^4\pi}\Big)\\
&&\ge{2^{-3}\epsilon_0}\quad \forall\ r\in [0,1)\ \ \&\ \ \xi\in F_j.\nonumber
\end{eqnarray}

If $\epsilon=2^{-4}\epsilon_0$ in (\ref{eYe}), then one can take such $s$ that $s_0<s<1$ and (\ref{eYe}) is true for $|\phi\circ\sigma_b(a)|>s$. Assuming $t_j\ge s$ and recalling that the definition of $E_j$ ensures
$$
|\phi\circ\sigma_{b_j}(r\xi)|\to|\phi\circ\sigma_{b_j}(\xi)|>t_j\quad\hbox{as}\ \ r\to 1\ \ \hbox{for\ each}\ \ \xi\in E_j.
$$
Of course, this last property is valid for arbitrarily chosen point $\xi_j\in F_j$. Note that
$|\phi\circ\sigma_{b_j}(0)|=|\phi(b_j)|\le s_0.$
Thus, by continuity of $|\phi\circ\sigma_{b_j}|$ there exists an $r_j\in (0,1)$ such that $|\phi\circ\sigma_{b_j}(r_j\xi_j)|=s$. If $a_j=r_j\xi_j$
then
$$
\rho\big(\phi\circ\sigma_{b_j}(\xi),\phi\circ\sigma_{b_j}(a_j)\big)\ge\rho(t_j,s)\quad\forall\quad \xi\in E_j,
$$
and hence (\ref{e6e}) and $q=1$ are applied to deduce
\begin{eqnarray*}
&&\frac{(1-|\sigma_{b_j}(a_j)|^2)^{1-q}}{(1-|\phi\circ\sigma_{b_j}(a_j)|^2)^{1-p}}\int_{J(a_j)}\rho\big(\phi\circ\sigma_{b_j}
(\xi),\phi\circ\sigma_{b_j}(a_j)\big)^2\,\frac{|d\xi|}{(2\pi)^{-1}|J(a_j)|}\\
&&\ge\left(\frac{(1-|b_j|^2)^{1-q}}{(1-|\phi(b_j)|^2)^{1-p}}\right)\left(\frac{|J(a_j)\cap E_j|}{|J(a_j)|}\right)\left(\frac{1-|\phi(b_j)|^2}{1-|\phi\circ\sigma_{b_j}(a_j)|^2}\right)^{1-p}\rho(t_j,s)^2\\
&&\ge {2^{-3}\epsilon_0}\left(\frac{\min\{1,(1-s_0^2)^{1-p}\}}{(1-s^2)^{1-p}}\right)\rho(t_j,s)^2.
\end{eqnarray*}
Since $\lim_{j\to\infty}\rho(t_j,s)=1$, it follows from (\ref{eYe}) that
\begin{eqnarray*}
&&0\\
&&=\lim_{j\to\infty}\frac{(1-|\sigma_{b_j}(a_j)|^2)^{1-q}}{(1-|\phi\circ\sigma_{b_j}(a_j)|^2)^{1-p}}\int_{J(a_j)}
\left(\frac{\rho\big(\phi\circ\sigma_{a_j}
(\xi),\phi\circ\sigma_{a_j}(b_j)\big)^2}{(2\pi)^{-1}|J(a_j)|}\right)\,|d\xi|\\
&&\ge 2^{-3}\epsilon_0\left(\frac{\min\{1,(1-s_0^2)^{1-p}\}}{(1-s^2)^{1-p}}\right),
\end{eqnarray*}
a contradiction. In other words, (\ref{eee}) must be true under (\ref{e2}) being valid.
\end{proof}

\begin{proof}[Proof of Theorem \ref{tB}] Suppose that $C_\phi:\mathcal{CA}_p\mapsto\mathcal{CA}_q$ is compact. Of course, this operator is bounded, and thus (\ref{e1}) holds. Choosing $b=\phi(a)$, we see that $f_b$ defined in Lemma \ref{l21} tends to $0$ uniformly on compact subsets of $\mathbb D$ whenever $|b|\to 1$. Thus, $\lim_{|b|\to 1}\|C_\phi f_b\|_{\mathcal{CA}_p,\ast}=0$. As an immediate by-product of the $\mathsf{C}$-part in the proof of Theorem \ref{tA}, we have
$$
0=\lim_{|b|\to 1}\|C_\phi f_b\|_{\mathcal{CA}_p,\ast}^2\gtrsim\lim_{|b|\to 1}\frac{(1-|a|^2)^{1-q}|b|^2}{(1-|b|^2)^{1-p}}\|\sigma_{b}\circ\phi\circ\sigma_a\|_2^2,
$$
whence deriving (\ref{e2}).

Next, we deal with the converse part of Theorem \ref{tB} according to $(p,q)\in [0,2)\times[1,1]$ and $(p,q)\in (1,2)\times[0,2)$. In order to verify that $C_\phi: \mathcal{CA}_p\mapsto\mathcal{CA}_q$ is a compact operator, it suffices to check that $\lim_{n\to\infty}\|C_\phi f_n\|_{\mathcal{CA}_q,\ast}=0$ holds for any sequence $(f_n)_{n=1}^\infty$ in $\mathcal{CA}_p$ with $\|f_n\|_{\mathcal{CA}_p,\ast}\le 1$ and $f_n\to 0$ on compact subsets of $\mathbb D$ as $n\to\infty$.

{\it Situation 1} - assume that (\ref{e1}) holds and (\ref{e2}) is valid for $(p,q)\in [0,2)\times[1,1]$. Upon writing
$$
\|C_\phi f_n\|_{\mathcal{CA}_q,\ast}^2\lesssim\sup_{|\phi(a)|>s}\mathsf{T}(n,a,q)+\sup_{|\phi(a)|\le s}\mathsf{T}(n,a,q),
$$
where
$$
0<s<1\ \&\ \mathsf{T}(n,a,q)=(1-|a|^2)^{1-q}\|f_n\circ\phi\circ\sigma_a-f_n\circ\phi(a)\|_2^2,
$$
we have to control $\sup_{|\phi(a)|>s}\mathsf{T}(n,a,q)$ and $\sup_{|\phi(a)|\le s}\mathsf{T}(n,a,q)$ from above. To do so, set
$$
\begin{cases}
f_{n,a}=f_n\circ\phi\circ\sigma_a-f_n(\phi(a));\\
g_{n,a}=f_n\circ\sigma_{\phi(a)}-f_n(\phi(a));\\
\psi_{a}=\sigma_{\phi(a)}\circ\phi\circ\sigma_a;\\
E(\phi,a,t)=\{\xi\in\mathbb T:\ |\sigma_{\phi(a)}\circ\phi\circ\sigma_a(\xi)|>t\}.
\end{cases}
$$
Using (\ref{e8}) we obtain
\begin{eqnarray*}
&&\sup_{|\phi(a)|>s}\mathsf{T}(n,a,q)\\
&&\approx\sup_{|\phi(a)|>s}(1-|a|^2)^{1-q}\|f_{n,a}\|_2^2\\
&&\lesssim\sup_{|\phi(a)|>s}(1-|a|^2)^{1-q}\|g_{n,a}\|_2^2\|\psi_a\|_2^2\\
&&\lesssim\sup_{|\phi(a)|>s}\frac{(1-|a|^2)^{1-q}}{(1-|\phi(a)|^2)^{1-p}}\|\psi_a\|_2^2\|f_n\|_{\mathcal{CA}_p,\ast}^2\\
&&\lesssim\sup_{|\phi(a)|>s}\frac{(1-|a|^2)^{1-q}}{(1-|\phi(a)|^2)^{1-p}}\|\psi_a\|_2^2,
\end{eqnarray*}
whence getting by (\ref{e2})
\begin{equation}\label{e7e}
\lim_{s\to 1}\sup_{|\phi(a)|>s}\mathsf{T}(n,a,q)=0\quad\forall\quad n=1,2,3,...
\end{equation}
Meanwhile,
$$
\sup_{|\phi(a)|\le s}\mathsf{T}(n,a,q)\lesssim
\sup_{|\phi(a)|\le s}\mathsf{T}_1(n,a,q)+\sup_{|\phi(a)|\le s}\mathsf{T}_2(n,a,q),
$$
where
$$
\begin{cases}
\mathsf{T}_1(n,a,q)=(1-|a|^2)^{1-q}\int_{\mathbb T\setminus E(\phi,a,t)}|f_{n,a}(\xi)|^2\,|d\xi|;\\
\mathsf{T}_2(n,a,q)=(1-|a|^2)^{1-q}\int_{E(\phi,a,t)}|f_{n,a}(\xi)|^2\,|d\xi|.
\end{cases}
$$
Applying Schwarz's lemma to $g_{n,a}$ or using \cite[(3.19)]{Lait1} we get
$$
\sup_{|z|\le t}{|z|^{-1}|g_{n,a}(z)|}\le 2\sup_{|w|\le t}|g_{n,a}(w)|
$$
thereby deriving
\begin{eqnarray*}
&&\sup_{|\phi(a)|\le s}\mathsf{T}_1(n,a,q)\\
&&\lesssim\sup_{|\phi(a)|\le s}(1-|a|^2)^{1-q}\sup_{|w|\le t}|g_{n,a}(w)|^2\int_{\mathbb T}|\psi_a(\xi)|^2\,|d\xi|\\
&&\lesssim\big(1+(1-s)^{p-1}\big)\sup_{|w|\le t}|g_{n,a}(w)|^2\sup_{|\phi(a)|\le s}\frac{(1-|a|^2)^{1-q}}{(1-|\phi(a)|^2)^{1-p}}\|\psi_a\|_2^2\\
&&\to 0\quad\hbox{as}\quad n\to\infty,
\end{eqnarray*}
in which $|\phi(a)|\le s$ and $|w|\le t$ have been used.
Also, a combination of (\ref{e8}), (\ref{e1}) and $q=1$ gives that if
$$
\begin{cases}
\lambda=(a\bar{b}-1)/(1-b\bar{a});\\
\tau=\phi\circ\sigma_a;\\
c=\sigma_b(a);\\
b\in\mathbb D,
\end{cases}
$$
then
\begin{eqnarray*}
&&\|f_{n,a}\|_{\mathcal{CA}_q,\ast}^2\\
&&=\sup_{b\in\mathbb D}(1-|b|^2)^{1-q}\|f_n\circ\tau\circ\sigma_b-f_n\circ\tau(b)\|_2^2\\
&&\lesssim\sup_{b\in\mathbb D}(1-|b|^2)^{1-q}\|f_n\circ\sigma_{\tau(b)}-f_n\circ\tau(b)\|_2^2\|\sigma_{\tau(b)}\circ\tau\circ\sigma_b\|_2^2\\
&&\lesssim\|f_n\|_{\mathcal{CA}_p,\ast}^2\sup_{b\in\mathbb D}\frac{(1-|b|^2)^{1-q}}{(1-|\tau(b)|^2)^{1-p}}\|\sigma_{\tau(b)}\circ\tau\circ\sigma_b\|_2^2\\
&&\lesssim\sup_{c\in\mathbb D}\frac{(1-|\lambda c|^2)^{1-q}}{(1-|\phi(\lambda c)|^2)^{1-p}}\|\sigma_{\phi(\lambda c)}\circ\phi\circ(\lambda\sigma_c)\|_2^2\\
&&\lesssim\sup_{c\in\mathbb D}\frac{(1-|c|^2)^{1-q}}{(1-|\phi(c)|^2)^{1-p}}\|\psi_c\|_2^2\\
&&<\infty,
\end{eqnarray*}
and hence from the Cauchy-Schwarz inequality, Lemmas \ref{l32}-\ref{l31} and $q=1$ it follows that
\begin{eqnarray*}
&&\sup_{|\phi(a)|\le s}\mathsf{T}_2(n,a,q)\\
&&\lesssim\sup_{|\phi(a)|\le s}{(1-|a|^2)^{1-q}}\left(\int_{E(\phi,a,t)}|f_{n,a}(\xi)|^4\,|d\xi|\right)^\frac12|E(\phi,a,t)|^\frac12\\
&&\lesssim\sup_{|\phi(a)|\le s}\left(\frac{(1-|a|^2)^{1-q}}{(1-|\phi(a)|^2)^{p-1}}\int_{\mathbb T}|f_{n,a}(\xi)|^4\,|d\xi|\right)^\frac12\left(\frac{(1-|a|^2)^{1-q}}{(1-|\phi(a)|^2)^{1-p}}|E(\phi,a,t)|\right)^\frac12\\
&&\lesssim\big(1+(1-s^2)^{1-p}\big)^\frac12\sup_{|\phi(a)|\le s}\left(\frac{\|f_{n,a}\|_2^2\|f_{n,a}\|_{\mathcal{CA}_q,\ast}^2}{(1-|a|^2)^{q-1}}
\right)^\frac12
\left(\frac{|E(\phi,a,t)|}{\frac{(1-|\phi(a)|^2)^{1-p}}{(1-|a|^2)^{1-q}}}\right)^\frac12\\
&&\lesssim\big(1+(1-s^2)^{1-p}\big)^\frac12\|f_{n,a}\|_{\mathcal{CA}_q,\ast}^2
\sup_{|\phi(a)|\le s}\left(\frac{(1-|a|^2)^{1-q}}{(1-|\phi(a)|^2)^{1-p}}|E(\phi,a,t)|\right)^\frac12\\
&&\lesssim\left(\frac{\sup_{c\in\mathbb D}\frac{(1-|c|^2)^{1-q}}{(1-|\phi(c)|^2)^{1-p}}\|\psi_c\|_2^2}{\big(1+(1-s^2)^{1-p}\big)^{-\frac12}}\right)
\sup_{|\phi(a)|\le s}\left(\frac{(1-|a|^2)^{1-q}}{(1-|\phi(a)|^2)^{1-p}}|E(\phi,a,t)|\right)^\frac12\\
&&\to 0\quad\hbox{as}\quad t\to 1.
\end{eqnarray*}
Consequently,
\begin{equation}\label{e8e}
\lim_{n\to\infty}\sup_{|\phi(a)|\le s}\mathsf{T}(n,a,q)=0.
\end{equation}
Putting (\ref{e7e}) and (\ref{e8e}) together, we reach $\lim_{n\to\infty}\|C_\phi f_n\|_{\mathcal{CA}_q,\ast}=0$.

{\it Situation 2} - assume that (\ref{e1}) holds and (\ref{e2}) is valid for $(p,q)\in (1,2)\times[0,2)$. Rewriting
\begin{eqnarray*}
&&\|C_\phi f_n\|_{\mathcal{CA}_q,\ast}^2\\
&&\lesssim\sup_{a\in\mathbb D}(1-|a|^2)^{1-q}\int_{\mathbb D}|f_n'(w)|^2N(\phi\circ\sigma_a,w)\,dm(w)\\
&&\le\sup_{a\in\mathbb D}\mathsf{U}(n,a,q,r)+\sup_{a\in\mathbb D}\mathsf{V}(n,a,q,r),
\end{eqnarray*}
where $2^{-1}\le r<1$ and
$$
\begin{cases}
\mathsf{U}(n,a,q,r)=(1-|a|^2)^{1-q}\int_{|\sigma_{\phi(a)}(w)|\le r}|f_n'(w)|^2N(\phi\circ\sigma_a,w)\,dm(w);\\
\mathsf{V}(n,a,q,r)=(1-|a|^2)^{1-q}\int_{|\sigma_{\phi(a)}(w)|>r}|f_n'(w)|^2N(\phi\circ\sigma_a,w)\,dm(w),
\end{cases}
$$
we have to control $\sup_{a\in\mathbb D}\mathsf{U}(n,a,q,r)$ and $\sup_{a\in\mathbb D}\mathsf{V}(n,a,q,r)$ for an appropriate $r\in [2^{-1},1)$. In the sequel, let $b=\phi(a)$.

{\it Sub-situation 1 - estimate for $\sup_{a\in\mathbb D}\mathsf{U}(n,a,q,r)$}. For this, we consider two cases for any given $s\in (0,1)$.

{\it Case $1_1$}: $|b|\le s$. Under this case, $|\sigma_b(w)|\le r$ ensures that $w$ belongs to a compact subset $K$ of $\mathbb D$, and therefore, it follows from $f_n\to 0$ on any compact subset of $\mathbb D$ and (\ref{e10}) that $\lim_{n\to\infty}\sup_{w\in K}|f_n'(w)|=0$ and consequently,
\begin{eqnarray*}
&&\lim_{n\to\infty}\sup_{|b|\le s}(1-|a|^2)^{1-q}\int_{|\sigma_b(w)|\le r}|f_n'(w)|^2N(\phi\circ\sigma_a,w)\,dm(w)\\
&&=\lim_{n\to\infty}\sup_{|b|\le s}\frac{(1-|a|^2)^{1-q}}{(1-|b|^2)^{1-p}}(1-|b|^2)^{p-1}\int_{|\sigma_b(w)|\le r}|f_n'(w)|^2N(\phi\circ\sigma_a,w)\,dm(w)\\
&&\lesssim\Big(\lim_{n\to\infty}\sup_{w\in K}|f_n'(w)|^2\Big)\sup_{|b|\le s}\frac{(1-|a|^2)^{1-q}}{(1-|b|^2)^{1-p}}\int_{|\sigma_b(w)|\le r}N(\phi\circ\sigma_a,w)\,dm(w)\\
&&\lesssim\Big(\lim_{n\to\infty}\sup_{w\in K}|f_n'(w)|^2\Big)\sup_{|b|\le s}\frac{(1-|a|^2)^{1-q}}{(1-|b|^2)^{1-p}}\int_{\mathbb D}N(\sigma_b\circ\phi\circ\sigma_a,z)\,dm(z)\\
&&\lesssim\Big(\lim_{n\to\infty}\sup_{w\in K}|f_n'(w)|^2\Big)\sup_{|b|\le s}\frac{(1-|a|^2)^{1-q}}{(1-|b|^2)^{1-p}}\|\sigma_b\circ\phi\circ\sigma_a\|_2^2\\
&&=0.
\end{eqnarray*}

{\it Case $1_2$}: $|b|>s$. Using (\ref{e5a}) we get
\begin{eqnarray*}
&&\sup_{|b|>s}(1-|a|^2)^{1-q}\int_{|\sigma_b(w)|\le r}|f_n'(w)|^2N(\phi\circ\sigma_a,w)\,dm(w)\\
&&\lesssim\|f_n\|^2_{\mathcal{CA}_p,\ast}\sup_{|b|>s}(1-|a|^2)^{1-q}\int_{|\sigma_b(w)|\le r}N(\sigma_b\circ\phi\circ\sigma_a,\sigma_b(w))\,\frac{dm(w)}{(1-|w|^2)^{3-p}}\\
&&\lesssim\sup_{|b|>s}(1-|a|^2)^{1-q}\int_{|z|\le r}\big(1-|\sigma_b(z)|^2\big)^{p-1}N(\sigma_b\circ\phi\circ\sigma_a,z)\,\frac{dm(z)}{(1-|z|^2)^{2}}
\\
&&\lesssim\sup_{|b|>s}\frac{(1-|a|^2)^{1-q}}{(1-|b|^2)^{1-p}}\int_{|z|\le r}N(\sigma_b\circ\phi\circ\sigma_a,z)\,\frac{dm(z)}{(1-|z|^2)^2}\\
&&\lesssim(1-r^2)^{-2}\sup_{|b|>s}\frac{(1-|a|^2)^{1-q}}{(1-|b|^2)^{1-p}}\|\sigma_b\circ\phi\circ\sigma_a\|_2^2\\
&&\to 0\quad\hbox{as}\quad s\to 1.
\end{eqnarray*}

Putting the above two cases together, we see that for any $\epsilon\in (0,1)$ there are two real numbers: $r_0\in [2^{-1},1)$; $s_0\in (0,1)$, and a natural number $n_0$ such that $n\ge n_0$
\begin{equation}\label{eUe}
\sup_{a\in\mathbb D}\mathsf{U}(n,a,q,r_0)\le\sup_{|b|\le s_0}\mathsf{U}(n,a,q,r_0)+\sup_{|b|> s_0}\mathsf{U}(n,a,q,r_0)<\epsilon.
\end{equation}

{\it Sub-situation 2 - estimate for $\sup_{a\in\mathbb D}\mathsf{V}(n,a,q,r)$}. Like {\it Sub-situation 1}, two treatments are required.

{\it Case $2_1$}: $|b|\le s$. For this case, we need the following by-product of \cite[Lemma 2.1]{Smith}: if $\psi$ is an analytic self-map of $\mathbb D$ with $\psi(0)=0$ then
\begin{equation}\label{esmith}
\mathsf{W}=\sup_{0<|w|<1}|w|^2N(\psi,w)<\infty\Longrightarrow\sup_{2^{-1}\le |w|<1}\frac{N(\psi,w)}{\ln|w|^{-1}}\le 4(\ln 2)^{-1}\mathsf{W}.
\end{equation}
Note that (\ref{e1}) and (\ref{e2}) imply respectively
\begin{equation}\label{e31ae}
\sup_{a\in\mathbb D}\frac{(1-|a|^2)^{1-q}}{(1-|b|^2)^{1-p}}\sup_{0<|w<1}|w|^2N(\sigma_b\circ\phi\circ\sigma_a,w)<\infty
\end{equation}
and
\begin{equation}\label{e31aee}
\lim_{|b|\to 1}\frac{(1-|a|^2)^{1-q}}{(1-|b|^2)^{1-p}}\sup_{0<|w<1}|w|^2N(\sigma_b\circ\phi\circ\sigma_a,w)=0
\end{equation}
thanks to the following (\ref{e9})-based mean value estimate for $N(\sigma_w\circ\sigma_b\circ\phi\circ\sigma_a,0)$ where $0<|w|<1$ (cf. \cite[(2.9)]{Lait1}):
\begin{eqnarray*}
&&|w|^2N(\sigma_b\circ\phi\circ\sigma_a,w)\\
&&=|w|^2N(\sigma_w\circ\sigma_b\circ\phi\circ\sigma_a,0)\\
&&\lesssim\int_{|z|<|w|}N(\sigma_w\circ\sigma_b\circ\phi\circ\sigma_a,z)\,dm(z)\\
&&\lesssim\int_{\mathbb D}N\big(\sigma_b\circ\phi\circ\sigma_a,\sigma_w(z)\big)\,dm(z)\\
&&\approx\int_{\mathbb D}|\sigma_w'(z)|^2N\big(\sigma_b\circ\phi\circ\sigma_a,z\big)\,dm(z)\\
&&\approx\|\sigma_w\circ\sigma_b\circ\phi\circ\sigma_a-\sigma_w\circ\sigma_b\circ\phi\circ\sigma_a(0)\|_2^2\\
&&\lesssim\|\sigma_{b}\circ\phi\circ\sigma_a\|_2^2.
\end{eqnarray*}
Thus, a combination of (\ref{esmith})-(\ref{e31ae})-(\ref{e31aee}) and H\"older's inequality gives
\begin{eqnarray*}
&&\sup_{|b|\le s}(1-|a|^2)^{1-q}\int_{|\sigma_b(w)|>r}|f_n'(w)|^2N(\phi\circ\sigma_a,w)\,dm(w)\\
&&\approx\sup_{|b|\le s}(1-|a|^2)^{1-q}\int_{|\sigma_b(w)|>r}|f_n'(w)|^2N\big(\sigma_{b}\circ\phi\circ\sigma_a,\sigma_b(w)\big)\,dm(w)\\
&&\lesssim\sup_{|b|\le s}(1-|b|^2)^{1-p}\int_{|\sigma_b(w)|> r}|f_n'(w)|^2N(\sigma_b,w)\,dm(w)\\
&&\lesssim\Big(1+(1-s^2)^{1-p}\Big)\int_{|\sigma_b(w)|>r}|f'_n(w)|^2N(\sigma_b,w)\,dm(w)\\
&&\lesssim\Big(1+(1-s^2)^{1-p}\Big)\int_{|z|>r}|(f_n\circ\sigma_b)'(z)|^2(1-|z|^2)\,dm(z)\\
&&\lesssim\Big(1+(1-s^2)^{1-p}\Big)\left(\frac{\int_{|z|>r}|(f_n\circ\sigma_b)'(z)|^4(1-|z|^2)^{4-p}\,dm(z)}
{\left(\int_{|z|>r}(1-|z|^2)^{p-2}\,dm(z)\right)^{-1}}\right)^{1/2}.
\end{eqnarray*}
Since $\|f_n\|_{\mathcal{CA}_p,\ast}\le 1$ and $|b|\le s<1$ ensure $\|f_n\circ\sigma_b\|_{\mathcal{CA}_p,\ast}\lesssim 1$, one concludes that $|(f_n\circ\sigma_b)'(z)|^2(1-|z|^2)dm(z)$ is $p$-Carleson measure with norm relying on $s$ and so that $d\mu_n(z)=|(f_n\circ\sigma_b)'(z)|^2(1-|z|^2)^{4-p}dm(z)$ is $3$-Carleson measure with norm relying on $s$. Now, it follows from \cite[Theorem 1.2]{Stegenga} that
\begin{eqnarray*}
&&\int_{|z|>r}|(f_n\circ\sigma_b)'(z)|^4(1-|z|^2)^{4-p}\,dm(z)\\
&&=\int_{|z|>r}|(f_n\circ\sigma_b)'(z)|^2\,d\mu_n(z)\\
&&\lesssim\|\mu_n\|_{\mathcal{CM}_3}\int_{\mathbb D}|(f_n\circ\sigma_b)'(z)|^2(1-|z|^2)\,dm(z)\\
&&\lesssim \|f_n\|_{\mathcal{CA}_p,\ast}^4\\
&&\lesssim 1.
\end{eqnarray*}
Note that
$$
\lim_{r\to 1}\int_{|z|>r}(1-|z|^2)^{p-2}\,dm(z)=0.
$$
So
$$
\lim_{r\to 1}\sup_{|b|\le s}(1-|a|^2)^{1-q}\int_{|\sigma_b(w)|>r}|f_n'(w)|^2N(\phi\circ\sigma_a,w)\,dm(w)=0
$$
holds for any $n=1,2,3,...$.

{\it Case $2_2$}: $|b|>s$. Since (\ref{e31aee}) implies that for any $\epsilon\in (0,1)$ there is an $s_0\in (0,1)$ such that
$$
|b|>s_0\Longrightarrow \frac{(1-|a|^2)^{1-q}}{(1-|b|^2)^{1-p}}\sup_{0<|w|<1}|w|^2N(\sigma_b\circ\phi\circ\sigma_a,w)<\epsilon.
$$
Thus, (\ref{esmith}) is applied once again to deduce that
\begin{eqnarray*}
&&N(\phi\circ\sigma_a,w)\\
&&=N\big(\sigma_b\circ\phi\circ\sigma_a,\sigma_b(w)\big)\\
&&\lesssim \frac{\epsilon(1-|b|^2)^{1-p}}{(1-|a|^2)^{1-q}}\ln|\sigma_b(w)|^{-1}\\
&&\approx\frac{\epsilon(1-|b|^2)^{1-p}}{(1-|a|^2)^{1-q}}N(\sigma_b,w)\\
&&\hbox{as}\quad |\sigma_b(w)|>r>2^{-1}.
\end{eqnarray*}
Consequently,
\begin{eqnarray*}
&&\sup_{|b|> s_0}(1-|a|^2)^{1-q}\int_{|\sigma_b(w)|>r}|f_n'(w)|^2N(\phi\circ\sigma_a,w)\,dm(w)\\
&&\lesssim \epsilon\sup_{|b|>s_0}(1-|b|^2)^{1-p}\int_{|\sigma_b(w)|>r}|f_n'(w)|^2N(\sigma_b,w)\,dm(w)\\
&&\lesssim\epsilon\|f_n\|^2_{\mathcal{CA}_p,\ast}\\
&&\lesssim\epsilon.
\end{eqnarray*}
The previous discussions on {\it Case $2_1$} and {\it Case $2_2$}  indicate
\begin{equation}\label{eVe}
\lim_{n\to\infty}\sup_{a\in\mathbb D}\mathsf{V}(n,a,q,r)=0.
\end{equation}
Obviously, (\ref{eUe}) and (\ref{eVe}) give
$\lim_{n\to\infty}\|C_\phi f_n\|_{\mathcal{CA}_q,\ast}=0$.

\end{proof}

\begin{proof}[Proof of Corollary \ref{cor2}] This follows from (\ref{e5}), Theorem \ref{tB} and \cite[Theorem 1.4 (c)]{Xiao}.
\end{proof}

\end{document}